\documentclass[12pt]{article}

\usepackage{euscript}
\usepackage{latexsym}
\usepackage{amsmath}
\usepackage{amssymb}
\usepackage{indentfirst}

\renewcommand{\iff}{if and only if}

\newcommand{\mI}{\mathcal{IC}}
\newcommand{\ov}{\overline}

\newenvironment{proof}{\rm \trivlist\item[\hskip \labelsep{\it
Proof:}]}{{\hfill \scriptsize $\blacksquare$} \endtrivlist}

\newtheorem{prop}{Proposition}[section]

\newtheorem{lem}[prop]{Lemma}
\newtheorem{cor}[prop]{Corollary}

\def\softd{{\leavevmode\setbox1=\hbox{d}\hbox
									to 1.15\wd1{d\kern-0.2ex{\char039}\hss}}}

\begin{document}

\title{On the question of embedding a semigroup into an idempotent generated one}
\author{Lu\'\i s Oliveira}
\date{}

\maketitle

\begin{abstract}
In this paper we present a new embedding of a semigroup into a semiband (idempotent-generated semigroup) of depth 4 (every ele\-ment is the product of 4 idempotents) using a semidirect product construction. Our embedding does not assume that $S$ is a monoid (although it assumes a weaker condition), and works also for (non-monoid) regular semigroups. In fact, this semidirect product is particularly useful for regular semigroups since we can defined another embedding for these semigroups into a smaller semiband of depth 2. In this paper we shall compare our construction with other known embeddings, and we shall see that some properties of $S$ are preserved by it.  
\end{abstract}

\section{Introduction}

The reader should consult \cite{howiebook} for undefined terms used in this paper. Let $S$ be a semigroup. As usual, we shall denote by $S^1$ the monoid induced by $S$, that is, $S^1=S$ if $S$ is a monoid, or otherwise $S^1$ is the semigroup $S$ with an extra identity element adjoined. We shall denote by $1$ the identity element of $S^1$. Also, we shall denote by $E(S)$ the set of idempotents of $S$ (we shall write only $E$ if no ambiguity occurs) and by $\langle E(S)\rangle$ the subsemigroup of $S$ generated by its idempotents. A semigroup $S$ is a semiband if $S=\langle E\rangle$. If $E^n$ denotes the set of elements of $S$ that are the product of $n$ idempotents, then
\[\langle E\rangle =\cup_{n=1}^{\infty} E^n \, .\]
Clearly $E^n\subseteq E^{n+1}$ since we are talking about products of idempotents, and if $E^n=E^{n+1}$ then $\langle E\rangle = E^n$. The depth of an element $a\in \langle E\rangle$ is the positive integer $k$ such that $a\in E^k\setminus E^{k-1}$ ($E^0=\emptyset$). We shall say that a semiband $S$ has depth $k$ if $E^{k-1}\neq E^k =E^{k+1}$, and has infinite depth if $E^n\neq E^{n+1}$ for all $n$. Note that finite semibands have always finite depth.

There have been described in the literature several ways of embedding a semigroup into a semiband. The first such known embedding is due to Howie \cite{howie66} and it embeds a semigroup $S$ into the (regular) subsemigroup $\langle E(\mathcal{T}_X)\rangle$ of some full transformation semigroup $\mathcal{T}_X$. Howie's embedding allows to embed a semigroup of order $n$ (that is, with $n$ elements) into a regular semiband of order at most $(n+3)^{n+3}-(n+3)!\,$. Another embedding also involving an idempotent generated subsemigroup of some full transformation semigroup appeared in \cite{benzakenmayr} where it was attributed to Perrot. This later embedding was investigated by Pastijn \cite{pastijn77} who proved that the mentioned semiband has depth 4, and if $S$ has order $n$, then the semiband has order at most $2n^2+4n+1$. However, unlike Howie's embedding, this semiband is regular only when $S$ is regular.

To answer to a question posed by Howie in \cite{howie81}, Laffey \cite{laffey83} and Hall presented each one a way to embed a finite semigroup into a finite semiband of depth 2 (the proof of the latter appears in \cite{giraldeshowie}). Laffey's embedding is based on linear algebra techniques, while Hall's embedding makes use of the Rees matrix semigroup construction. Higgins \cite{higgins95} presented another proof that every semigroup can be embedded into a regular semiband of depth 2 in a way that preserves finiteness. Higgins' proof relies on embedding a full transformation semigroup $\mathcal{T}_X$ into a regular subsemiband of depth 2 of $\mathcal{T}_{X\cup X'}$ where $X'$ is another set disjoint from $X$ but with the same cardinality. 

In \cite{pastijnyan}, Pastijn and Yan introduced another way to embed a semigroup into a semiband of depth 4. This new semiband is given by a presentation and it follows closed to the results obtained in \cite{pastijn77}. The semiband used in \cite{pastijn77} can be seen as a quotient of the semiband presented in \cite{pastijnyan}. Petrich \cite{petrich04} presented another idempotent generated Rees matrix semigroup construction in which we can embed a given semigroup. This idempotent generated Rees matrix semigroup is isomorphic to the semigroup presented in \cite{pastijnyan}.

Recently, Almeida and Moura \cite{almeidamoura} studied the question of embedding a semigroup into a semiband inside some pseudovarieties of semigroups using some known embeddings, but also some profinite techniques. As an example, they were able to show that every finite semigroup $S$ whose regular $\mathcal{D}$-classes are aperiodic semigroups is embeddable into a semiband with the same cha\-racteristic using profinite techniques that do not give us an explicit semiband construction. All the known embeddings of a semigroup into a semiband do not preserve this property of $S$, and we must say that the new embeddings presented in this paper continue to fail to preserve this property of $S$. So, the question of finding an explicit semiband construction that allows to embed a semigroup whose regular $\mathcal{D}$-classes are aperiodic semigroups into a semiband with the same characteristic is still open.

In the following section we shall recall some of the embedding described above. In Section 3 we shall present a new embedding. We shall embed a semigroup $S$ with a certain property (which includes both monoids and regular semigroups) into a semiband $T(S)$ of depth 4 that is a subsemigroup of a semidirect product of $S\times S$ by $R_2$, where $R_2$ denotes the two element right-zero semigroup. Note that this embedding will work for any semigroup $S$ even if it does not have the referred property since we can always consider the monoid $S^1$ instead. The properties of $S$ preserved by $T(S)$ is the focus of study of Section 4.

The semidirect product introduced in Section 3 is particularly interesting if we consider only regular semigroups. If $S$ is a regular semigroup, we will embed $S$ into another subsemiband of the referred semidirect product that has only depth 2. This will be the topic of study of Section 5.

\section{Recalling some embeddings}\label{section2}

In this section we shall briefly describe some known embeddings of semigroups into semibands, but we will not do it chronologically. We begin with the Pastijn and Yan's embedding \cite{pastijnyan}. Let $\ov{X}=\{\ov{h}\}\cup\{\ov{s}\,:\;s\in S^1\}$ and consider the semigroup $F(S)$ given by the following presentation
\[F(S)=\langle\, \ov{X}\, |\; \ov{x}^2=\ov{x}\, ,\; \ov{h}\ov{1}=\ov{1}\, ,\; \ov{1}\ov{h}=\ov{h} \, ,\; \ov{s}\ov{t}=\ov{s} \, ,\; \ov{h}\ov{s}\ov{h}\ov{t}=\ov{h}\ov{st}\,\rangle\, ,\]
where $\ov{x}\in\ov{X}$ and $s,t\in S^1\,$. Then $F(S)$ is generated by the set of idempotents $\ov{X}$ and by \cite[Lemma 16]{pastijnyan}, the list
\[\ov{h}\, ,\;\ov{1} \, ,\;\ov{s} \, ,\;\ov{h}\ov{s} \, ,\;\ov{s}\ov{h} \, ,\;\ov{h}\ov{s}\ov{h} \, ,\; \ov{s}\ov{h}\ov{t} \, ,\; \ov{s}\ov{h}\ov{t}\ov{h}\, ,\qquad s,t\in S^1\setminus\{1\}\, ,\]
presents all elements of $F(S)$ (and the elements of this list are all pairwise distinct). Further, since $\ov{h}=\ov{1}\ov{h}\ov{1}\ov{h}$, $\ov{1}=\ov{1}\ov{h} \ov{1}$, $\ov{s}=\ov{s}\ov{h}\ov{1}$, $\ov{h}\ov{s}= \ov{1}\ov{h}\ov{s}$, $\ov{s}\ov{h}=\ov{s}\ov{h} \ov{1}\ov{h}$ and $\ov{h}\ov{s}\ov{h}=\ov{1}\ov{h}\ov{s} \ov{h}$, we conclude that
\[F(S)=\{\ov{s}\ov{h}\ov{t}\,,\;\ov{s}\ov{h}\ov{t}\ov{h}\, :\; s,t\in S^1\}\]
with any two elements of this set distinct from each other.

The mapping $S\rightarrow F(S)\, ,\; s\rightarrow \ov{h}\ov{s}=\ov{1}\ov{h}\ov{s}$ clearly embeds $S$ into the semiband $F(S)$ of depth 4. Note that if $S$ has order $n$ and it is a monoid, then $F(S)$ has order $2n^2$; but if is not a monoid, then $S$ is embeddable into the semiband $F(S)\setminus\{\ov{1}\}$ that has order $2n^2+4n+1$.   

The semiband used in \cite{pastijn77} to embed a semigroup into a semiband can be seen as a quotient of the semigroup $F(S)$. We can easily check that 
\[\rho=\{(\ov{s}\ov{h}\ov{t},\ov{r}\ov{h}\ov{t}),\, (\ov{s}\ov{h}\ov{t}\ov{h},\ov{r}\ov{h}\ov{t}\ov{h}):\, st=rt\,\mbox{ for }\; r,s,t\in S^1\}\]
is a congruence on $F(S)$. Let $A(S)=F(S)/\rho$. This quotient semigroup is isomorphic to the semigroup of \cite{pastijn77} denoted by $\widetilde{\mathcal{A}(S)}$. Clearly the mapping $S\rightarrow A(S),\; s\rightarrow \ov{h}\ov{s}\rho$ embeds $S$ into $A(S)$ since $\ov{h}\ov{s}\rho= \ov{h}\ov{r}\rho$ \iff\ $s=r$. However, in \cite{pastijn77}, it was used instead the following embedding: $S\rightarrow A(S)\, ,\; s\rightarrow \ov{h}\ov{s}\ov{h}\rho$.

Let $\Sigma=\{\sigma,\tau\}$ and consider the Rees matrix semigroup 
\[\Phi S=\mathcal{M}(S^1,S^1,\Sigma; Q)\]
where $Q=(q_{\alpha s})$ is the $\Sigma\times S^1$-matrix defined by $q_{\sigma s}=1$ and $q_{\tau s}=s$ for $s\in S^1$. By \cite[Theorem 3.1]{petrich04} there is an isomorphism $\chi$ from $F(S)$ onto $\Phi S$ such that $\ov{h}\chi =(1,1,\tau)$ and $\ov{s}\chi=(1,s,\sigma)$ for $s\in S^1$. Thus $\Phi S$ is a semiband of depth 4. Clearly the mapping $S\rightarrow \Phi S,\; s\rightarrow (1,s,\sigma)$ embeds $S$ into $\Phi S$. 

The last embedding we shall recall is the one introduced by Higgins \cite{higgins95} where the full transformation semigroup $\mathcal{T}_X$ is embedded into a subsemiband of the full transformation semigroup $\mathcal{T}_Y$ for $Y=X\cup X'$ and $X'$ a set disjoint from $X$ but with the same cardinality. Let $x\rightarrow x'$ be a bijection from $X$ onto $X'$ and, for each $\alpha\in\mathcal{T}_X$, define $\alpha'\in\mathcal{T}_Y$ by $x\alpha'=x\alpha=x'\alpha'$ for $x\in X$. Let also
\[T=\{\alpha\in\mathcal{T}_Y\, :\; X\alpha=X'\alpha\subseteq X \mbox{ or } X\alpha=X'\alpha\subseteq X'\,\}\, .\]
It was shown in \cite{higgins95} that $T$ is a regular subsemiband of $\mathcal{T}_Y$ of depth 2 and that the mapping $\mathcal{T}_X\rightarrow T,\; \alpha\rightarrow \alpha'$ embeds $\mathcal{T}_X$ into $T$.

\section{A new embedding}\label{section3}

Let $R$ and $T$ be two semigroups. A (left) action of $R$ on $T$ is a mapping $R\times T\rightarrow T,\; (a,t)\rightarrow {}^at$ 
such that ${}^a(ts)= {}^at\cdot {}^as$ and ${}^{ab}t= {}^a({}^bt)$ for all $a,b\in R$ and all $t,u\in T$. The semidirect product $T*R$ of $T$ by $R$ with respect to this action is the semigroup obtained by defining on $T\times R$ the product
\[(t,a)\cdot (u,b)=(t\cdot {}^au,ab)\qquad (a,b\in R,\; t,u\in T)\, .\]

Let $R_2=\{\sigma,\tau\}$ be the two element right-zero semigroup and let $T=S\times S$. Consider the mapping $R_2\times T\rightarrow T$ defined by
\[{}^\sigma (t,s)=(s,s)\quad\mbox{ and }\quad {}^\tau (t,s)=(t,t)\, .\]
Clearly this mapping is an action of $R_2$ on $T=S\times S$. Consider the semidirect product $T*R_2$ induced by this action. We shall look at the elements of $T*R_2$ as triples from $S\times S\times R_2$. Then
\[(t_1,s_1,\alpha_1)(t_2,s_2,\alpha_2)=\left\{\begin{array}{ll}
(t_1s_2\, ,\, s_1s_2\, ,\, \alpha_2) &\mbox{ if } \alpha_1=\sigma \\ [.2cm]
(t_1t_2\, ,\, s_1t_2\, ,\, \alpha_2) &\mbox{ if } \alpha_1=\tau
\end{array}\right.\]
for $t_1,t_2,s_1,s_2\in S$ and $\alpha_1,\alpha_2\in R_2$.

\begin{lem}\label{etr2}
$E(T*R_2)=\{(s,e,\sigma)\, ,\; (e,s,\tau)\,:\; e\in E(S) ,\, se=s\}\,$.
\end{lem}

\begin{proof}
Clearly, since $(t,u,\sigma)^2=(tu,u^2,\sigma)$, $(t,u,\sigma)\in E(T*R_2)$ \iff\ $tu=t$ and $u\in E(S)$. Similarly, $(t,u,\tau)\in E(T*R_2)$ \iff\ $t\in E(S)$ and $ut=u$ because $(t,u,\tau)^2=(t^2,ut,\tau)$.
\end{proof}

For the purpose of this paper, we shall say that an element $s\in S$ is idempotent covered in $S$ if $se=s=fs$ for some $e,f\in E(S)$, that is, if $s\in Se\cap fS$ for some $e,f\in E(S)$. We shall say that a semigroup $S$ is idempotent covered if all elements of $S$ are idempotent covered in $S$. Note that all monoids and all regular semigroups are idempotent covered semigroups and that $S$ is an idempotent covered semigroup \iff\
\[S=\cup_{e\in E(S)} Se \quad\mbox{ and }\quad S=\cup_{e\in E(S)} eS\, ,\]
that is, \iff\ $S$ is both the union of its  idempotent generated principal left ideals and the union of its idempotent generated principal right ideals. We shall denote by $\mI$ the class of all idempotent covered semigroups.

Note that if $S\in\mI$ then the set $E(T*R_2)$ is not empty. If $s\in Se$ for some $e\in E(S)$, then $(e,s,\tau)$ and $(s,e,\sigma)$ are idempotents of $T*R_2$; and every idempotent of $T*R_2$ is obtained in this manner. Consider now the subset of idempotents 
\[E=\{(s,t,\alpha)\in E(T*R_2)\, :\; s\in S^1t\}\, ,\] 
and let $T(S)=\langle E\rangle$.

\begin{prop}
Let $S\in\mI$. Then:
\begin{itemize}
\item[$(i)$] $E=\{(s,e,\sigma):\, e\in E(S),\, se=s\}\cup\{(e,s,\tau) :\, e\in E(S),\, e\mathcal{L} s\}$.
\item[$(ii)$] $T(S)=\{(s,t,\alpha)\in T*R_2 :\, s\in S^1t\}$ is a semiband of depth 4 with $E(T(S))=E$.
\item[$(iii)$] The mapping $\varphi: S\rightarrow T(S) ,\, s\rightarrow (s,s,\sigma)$ embeds $S$ into $T(S)$.
\end{itemize}
\end{prop}

\begin{proof}
$(i)$. By Lemma \ref{etr2} the elements of $\{(s,e,\sigma) :\, e\in E(S),\, se=s\}\cup\{(e,s,\tau) :\, e\in E(S),\, e\mathcal{L} s\}$ are idempotents of $T*R_2$ that clearly belong to $E$. Further, if $(a,b,\sigma)\in E\subseteq E(T*R_2)$, then $b\in E(S)$ and $ab=a$. Finally, let $(a,b,\tau)\in E$. Then $a\in S^1b$ and once more by Lemma \ref{etr2}, $a\in E(S)$ and $ba=b$, whence $a\mathcal{L} b$. We have proved $(i)$.

$(ii)$. It is trivial to verify that $A=\{(s,t,\alpha)\in T*R_2 :\, s\in S^1t\}$ is a subsemigroup of $T*R_2$ containing $E$, and so it contains $T(S)$ also. Let $s,t\in S$ such that $s\in S^1t$. Then let $e,f\in E(S)$ and $a\in S^1$ such that $t\in Se\cap fS$ and $s=at$. Note that $(af,f,\sigma)$, $(f,f,\tau)$, $(t,e,\sigma)$ and $(e,e,\tau)$ are idempotents from $E$ such that
\[(s,t,\sigma)=(af,f,\sigma)(f,f,\tau)(t,e,\sigma)\quad\mbox{ and }\quad (s,t,\tau)=(s,t,\sigma)(e,e,\tau)\,\]
Thus $(s,t,\sigma),(s,t,\tau)\in T(S)$ and $T(S)=A$ is a semiband of depth 4. Finally, by the definition of $E$ and since $T(S)=A$, it is obvious that $E(T(S))=E$.

$(iii)$. The mapping $\varphi$ is well defined by $(ii)$ and it is now obvious that it embeds $S$ into $T(S)$.
\end{proof}

Note that if $S$ is not idempotent covered, then we can consider $S^1$ ins\-tead of $S$ and embed $S$ into $T(S^1)$. Thus we have another proof that every semigroup can be embedded into a semiband.

The following result compares the two semibands $T(S)$ and $A(S)$ of depth 4 in which we can embed $S\in\mI$.

\begin{prop}
$T(S)$ and $A(S)$ are isomorphic for every monoid $S$. If $S\in\mI$ is not a monoid, then $T(S)$ is isomorphic to the subsemigroup $A_1=\{(\ov{h}\ov{s})\rho,\, (\ov{h}\ov{s}\ov{h})\rho,\, (\ov{r}\ov{h}\ov{s})\rho,\, (\ov{r}\ov{h}\ov{s}\ov{h})\rho \,:\; r,s\in S\}$ of $A(S)$.
\end{prop}

\begin{proof}
Define $\psi:T(S^1)\rightarrow A(S)$ as follows:
\[(s,t,\sigma)\psi=(\ov{a}\ov{h}\ov{t})\rho \quad\mbox{ and }\quad (s,t,\tau)\psi= (\ov{a}\ov{h}\ov{t}\ov{h})\rho\, ,\]
where $a\in S^1$ is such that $s=at$. Note that if $s=bt$ for some $b\in S^1$, then $(\ov{a}\ov{h}\ov{t})\rho= (\ov{b}\ov{h}\ov{t})\rho$ and $(\ov{a}\ov{h}\ov{t}\ov{h})\rho= (\ov{b}\ov{h}\ov{t} \ov{h})\rho$ by definition of $\rho$, and so $\psi$ is well defined. Further, the description of $A(S)$ also allows us to conclude immediately that $\psi$ is a bijection. Finally, to prove that $\psi$ is an isomorphism we must check that
\[((s_1,t_1,\alpha_1)(s_2,t_2,\alpha_2))\psi = (s_1,t_1,\alpha_1)\psi\cdot(s_2,t_2,\alpha_2)\psi\]
for any $(s_1,t_1,\alpha_1),\,(s_2,t_2,\alpha_2)\in T(S^1)$. We have four straightforward cases depending on whether $\alpha_1$ and $\alpha_2$ are $\sigma$ and/or $\tau$ that we leave to the reader to verify.
 
We can now conclude that $T(S)$ and $A(S)$ are isomorphic if $S$ is a monoid. If $S\in\mI$ is not a monoid, then $T(S)=\{(s,t,\alpha)\in T(S^1) \,:\; t\in S\}$. Observe now that for $t\in S$,
\[(s,t,\sigma)\psi=\left\{\begin{array}{ll}
(\ov{h}\ov{t})\rho & \mbox{ if } s=t \\ [.2cm]
(\ov{r}\ov{h}\ov{t})\rho & \mbox{ if } s\neq t \mbox{ and } s=rt \; ,\end{array}\right.\]
and
\[(s,t,\tau)\psi=\left\{\begin{array}{ll}
(\ov{h}\ov{t}\ov{h})\rho & \mbox{ if } s=t \\ [.2cm]
(\ov{r}\ov{h}\ov{t}\ov{h})\rho & \mbox{ if } s\neq t \mbox{ and } s=rt \; .\end{array}\right.\]
So $(T(S))\psi =A_1$ and $A_1$ is a subsemigroup of $A(S)$ isomorphic to $T(S)$.
\end{proof}

Note that the inverse isomorphism $\psi^{-1}:A(S) \longrightarrow T(S^1)$ is defined by
\[((\ov{a}\ov{h}\ov{t})\rho)\psi^{-1}=(at,t,\sigma) \quad\mbox{ and }\quad ((\ov{a}\ov{h}\ov{t}\ov{h}) \rho)\psi^{-1}=(at,t,\tau)\, .\]
We will mention $\psi^{-1}$ again in the last section of this paper.

\section{Properties preserved by $T(S)$}

In this section we shall see that some properties of $S\in\mI$ are preserved by $T(S)$.

\begin{prop}\label{finperreg}
A semigroup $S\in\mI$ is finite, periodic or regular \iff\ the semiband $T(S)$ is respectively finite, periodic or regular.
\end{prop}

\begin{proof}
The finiteness case is obvious. Let $s,t\in S$ such that $s\in S^1t$. Since $\langle (s,t,\sigma)\rangle=\{(st^{n-1},t^n,\sigma)\, :\; n\geq 1\}$, we know that $\langle (s,t,\sigma)\rangle$ is finite \iff\ $\langle t\rangle$ is finite. Similarly, since $\langle (s,t,\tau)\rangle= \{(s^n,ts^{n-1},\tau)\, :\; n\geq 1\}$, we know that $\langle (s,t,\tau)\rangle$ is finite \iff\ $\langle s\rangle$ is finite. Thus $S$ is periodic \iff\ $T(S)$ is periodic. For the regularity case, observe that $tt't=t$ for some $t'\in S$ \iff\ $(s,t,\alpha) (t',t',\sigma)(s,t,\alpha)=(s,t,\alpha)$ for $\alpha\in R_2$; and so $S$ is regular \iff\ $T(S)$ is regular. 
\end{proof}

The next lemma will be useful to characterize the Green's relations on $T(S)$.

\begin{lem}\label{pre-green}
Let $S\in\mI$ and let $(s,t,\alpha),\,(u,v,\beta)\in T(S)$ such that $u=av$ for some $a\in S^1$.
\begin{itemize}
\item[$(i)$] $(s,t,\tau)\,\mathcal{R}\,(s,t,\sigma)$.
\item[$(ii)$] If $(s,t,\alpha)\,\mathcal{L}\,(u,v,\beta)$, then $\alpha=\beta$.
\item[$(iii)$] $(s,t,\tau)\,\mathcal{L}\,(u,v,\tau)$ \iff\ $(s,t,\sigma)\,\mathcal{L}\,(u,v,\sigma)$.
\item[$(iv)$] $(s,t,\sigma)\,\mathcal{L}\,(u,v,\sigma)$ \iff\ there exist $(a,b,\sigma),\,(c,d,\sigma)\in T(S)$ such that $(u,v,\sigma)=(a,b,\sigma)(s,t,\sigma)$ and $(s,t,\sigma)= (c,d,\sigma)(u,v,\sigma)$.
\end{itemize}
\end{lem}

\begin{proof}
The proof of $(i)$ is trivial since if $e\in E(S)$ is such that $te=t$, then $(s,t,\tau)(e,e,\sigma)=(s,t,\sigma)$ and $(s,t,\sigma)(e,e,\tau)=(s,t,\tau)$. Statement $(ii)$ is also obvious since $(a,b,\gamma)(s,t,\alpha)=(u,v,\beta)$ only occurs if $\alpha=\beta$. Statement $(iii)$ follows from the fact that $(a,b,\alpha)(a_1,b_1,\tau)=(a_2,b_2,\tau)$ \iff\ $(a,b,\alpha)(a_1,b_1,\sigma)=(a_2,b_2,\sigma)$. Finally, the statement $(iv)$ holds true since $(a,b,\tau)(c,d,\alpha) = (ac_1,bc_1,\sigma)(c,d,\alpha)$ for $c=c_1d$ in $S^1$.
\end{proof}

The natural partial order $\leq$ on a semigroup $S$ is defined by
\[s\leq t \quad\mbox{ \iff\ }\quad s=at=tb=sb \;\mbox{ for some } a,b\in S\, .\]
This natural partial order was introduced by Mitsch \cite{mitsch86} for any semigroup and it generalizes the more usual natural partial order for regular semigroups introduced independently by Hartwig \cite{hartwig80} and Nambooripad \cite{nambooripad80}. The next result describes the natural partial order and the Green's relations on $T(S)$. We shall use the notation $a(s,t)$ to represent the pair $(as,at)$ for $a\in S^1$ and $(s,t)\in S\times S$.

\begin{prop}\label{greenTS}
Let $S\in \mI$ and $(s,t,\alpha),\, (u,v,\beta)\in T(S)$. 
\begin{itemize}
\item[$(i)$] $(s,t,\alpha)\,\mathcal{R}\,(u,v,\beta)$ \iff\ $t\mathcal{R} v$ and $(s,u)=a(t,v)$ for some $a\in S^1$. 
\item[$(ii)$] $(s,t,\alpha)\,\mathcal{L}\,(u,v,\beta)$ \iff\ $\alpha=\beta$ and $t\mathcal{L}v$.
\item[$(iii)$] $(s,t,\alpha)\,\mathcal{H}\,(u,v,\beta)$ \iff\ $\alpha=\beta$, $t\mathcal{H}v$ and $(s,u)=a(t,v)$ for some $a\in S^1$.
\item[$(iv)$] $(s,t,\alpha)\,\mathcal{D}\,(u,v,\beta)$ \iff\ $t\mathcal{D}v$.
\item[$(v)$] $(s,t,\alpha)\,\mathcal{J}\,(u,v,\beta)$ \iff\ $t\mathcal{J}v$.
\item[$(vi)$] $(s,t,\alpha)\,\leq\,(u,v,\beta)$ \iff\ $\alpha=\beta$, $t\leq v$ and $(s,u)=a(t,v)$ for some $a\in S^1$.
\end{itemize}
\end{prop}

\begin{proof}
$(i)$. Assume $(s,t,\alpha)\,\mathcal{R}\,(u,v,\beta)$. Then  $(s,t,\sigma)\,\mathcal{R}\, (u,v,\sigma)$ by Lemma \ref{pre-green}.$(i)$ and there exist $(u_1,v_1,\sigma), (s_1,t_1,\sigma)\in T(S)$ such that
\[(u,v,\sigma)=(s,t,\sigma)(u_1,v_1,\sigma)=(sv_1,tv_1,\sigma)\]
and
\[(s,t,\sigma)=(u,v,\sigma)(s_1,t_1,\sigma)=(ut_1,vt_1,\sigma)\, .\]
Thus $v\mathcal{R}t$. If $a\in S^1$ is such that $s=at$, then $u=sv_1=atv_1=av$.

Assume now that $t\mathcal{L} v$ and $(s,u)=a(t,v)$ for some $a\in S^1$. Let $v_1,t_1\in S^1$ such that $t=vt_1$ and $v=tv_1$. Then
\[(s,t,\sigma)(v_1,v_1,\sigma)=(sv_1,tv_1,\sigma)=(atv_1,tv_1, \sigma)=(av,v,\sigma)=(u,v,\sigma)\]
and
\[(u,v,\sigma)(t_1,t_1,\sigma)=(ut_1,vt_1,\sigma)=(avt_1,vt_1, \sigma)=(at,t,\sigma)=(s,t,\sigma)\, .\]
Again by Lemma \ref{pre-green}.$(i)$ we conclude that $(s,t,\alpha)\,\mathcal{R}\, (u,v,\beta)$.

$(ii)$. Assume $(s,t,\alpha)\,\mathcal{L}\,(u,v,\beta)$. Then $\alpha=\beta$ and $(s,t,\sigma)\,\mathcal{L}\, (u,v,\sigma)$ by $(ii)$ and $(iii)$ of Lemma \ref{pre-green}, and by $(iv)$ of the same lemma there exist $(u_1,v_1,\sigma), (s_1,t_1,\sigma)\in T(S)$ such that
\[(u,v,\sigma)=(u_1,v_1,\sigma)(s,t,\sigma)=(u_1t,v_1t,\sigma)\]
and
\[(s,t,\sigma)=(s_1,t_1,\sigma)(u,v,\sigma)=(s_1v,t_1v,\sigma)\, .\]
Thus $t\mathcal{L}v$.

Assume now that $\alpha=\beta$ and $t\mathcal{L}v$. Let $a,b,s_1,u_1\in S^1$ such that $v=at$, $t=bv$, $s=s_1t$ and $u=u_1v$. Then
\[(u_1a,a,\sigma)(s,t,\sigma)=(u_1at,at,\sigma)=(u,v,\sigma)\]
and
\[(s_1b,b,\sigma)(u,v,\sigma)=(s_1bv,bv,\sigma)=(s,t,\sigma)\, ,\]
and by Lemma \ref{pre-green}.$(ii)$ and $(iii)$ we conclude that $(s,t,\alpha)\,\mathcal{L}\,(u,v,\beta)$.

$(iii)$. The statement $(iii)$ follows now from $(i)$ and $(ii)$.

$(iv)$. Assume $(s,t,\alpha)\,\mathcal{D}\,(u,v,\beta)$ and let $(s_1,t_1,\alpha_1)\in T(S)$ such that
\[(s,t,\alpha)\,\mathcal{R}\, (s_1,t_1,\alpha_1)\, \mathcal{L} \, (u,v,\beta)\, .\] 
By $(i)$ and $(ii)$ we must have $t\mathcal{R} t_1\mathcal{L} v$, and so $t\mathcal{D} v$.

Conversely, assume $t\mathcal{D} v$ and let $t_1\in S$ such that $t\mathcal{R} t_1\mathcal{L} v$. Let $s=at$ for some $a\in S^1$. Observe now that $(s,t,\alpha)\,\mathcal{R}\, (at_1,t_1,\beta)$ by $(i)$, and that $(at_1,t_1,\beta)\,\mathcal{L}\, (u,v,\beta)$ by $(ii)$. Hence
$(s,t,\alpha)\,\mathcal{D}\, (u,v,\beta)$.

$(v)$. By $(iv)$ it is enough to show that $(t,t,\sigma)\, \mathcal{J}\,(v,v,\sigma)$ \iff\ $t\mathcal{J} v$. But this is obvious since $a=bc$ for $a,b,c\in S$ \iff\ $(a,a,\sigma)=(b,b,\sigma)(c,c,\sigma)$.

$(vi)$. Assume $(s,t,\alpha)\,\leq\,(u,v,\beta)$ and let $(a,b,\gamma),\,(c,d,\delta)\in T(S)$ such that
\[(s,t,\alpha)=(a,b,\gamma)(u,v,\beta)=(u,v,\beta)(c,d,\delta) =(s,t,\alpha)(c,d,\delta)\, .\]
Then $\alpha=\beta=\delta$. We shall assume that $\alpha=\sigma$ and prove only this case since the case $\alpha=\tau$ is shown similarly. Then $t=vd=td$ and $s=ud=sd$. Let $u_1\in S^1$ such that $u=u_1v$. Then $s=ud=u_1vd=u_1t$ and $(s,u)=u_1(t,v)$. If $\gamma=\sigma$, then $t=bv$ and $t\leq v$. If $\gamma=\tau$, then $t=bu=bu_1v$ and $t\leq v$ again. We have shown the direct implication.

Assume now that $\alpha=\beta$, $t\leq v$ and $(s,u)=a(t,v)$, and let $b,c\in S^1$ such that $t=bv=vc=tc$. Once again we have two cases to consider, $\alpha=\sigma$ and $\alpha=\tau$, but since they are similar we shall prove only one. Thus assume that $\alpha=\sigma$. Then
\[\begin{array}{l}
(ab,b,\sigma)(u,v,\sigma)=(abv,bv,\sigma)=(at,t,\sigma)=(s,t,\sigma)\, , \\ [.3cm]
 (u,v,\sigma)(c,c,\sigma)=(uc,vc,\sigma)=(avc,vc,\sigma)= (at,t,\sigma)=(s,t,\sigma)\, ,\\ [.3cm]
(s,t,\sigma)(c,c,\sigma)=(sc,tc,\sigma)=(atc,tc,\sigma)= (at,t,\sigma)=(s,t,\sigma)\, ,
\end{array}\]
and $(s,t,\sigma)\leq (u,v,\sigma)$.
\end{proof}

The description of the Green's relations given by the previous result allows us to immediately conclude the following:

\begin{cor}\label{greenrestriction}
Let $S\in \mI$ and let $\varphi$ be the embedding of $S$ into $T(S)$ considered earlier. The restriction to $S\varphi$ of a Green's relation on $T(S)$ gives us precisely the corresponding Green's relation on $S\varphi$; in other words, if $\mathcal{K}\in\{ \mathcal{H},\mathcal{R},\mathcal{L},\mathcal{D},\mathcal{J}\}$ and $a\in S\varphi$, then
\[K_a^{T(S)}\cap S\varphi =K_a^{S\varphi}\, .\]
Further, $H_a^{T(S)}=H_a^{S\varphi}$ and each $\mathcal{J}$-class [$\mathcal{D}$-class] of $T(S)$ contains exactly one $\mathcal{J}$-class [$\mathcal{D}$-class] of $S\varphi$.
\end{cor}

The local submonoids of a semigroup $S$ are the subsemigroups $eSe$ of $S$ with $e\in E(S)$. In the next result we show that each maximal subgroup and each local submonoid of $S\in\mI$ is respectively isomorphic to some maximal subgroup and some local submonoid of $T(S)$, and vice-versa.

\begin{prop}
Let $S\in \mI$ and $e=(s,t,\alpha)\in E(T(S))$.
\begin{itemize}
\item[$(i)$] If $\alpha=\sigma$, then $\psi: H_e \longrightarrow H_t,\;(s_1,t_1,\sigma) \longmapsto t_1$ is an isomorphism from the (maximal) subgroup $H_e$ of $T(S)$ onto the (maximal) subgroup $H_t$ of $S$; if $\alpha=\tau$, then $\psi: H_e \longrightarrow H_s,\;(s_1,t_1,\tau) \longmapsto s_1$ is an isomorphism from the (maximal) subgroup $H_e$ of $T(S)$ onto the (maximal) subgroup $H_s$ of $S$.
\item[$(ii)$] If $\alpha=\sigma$, then $\chi: e\,T(S)\,e \longrightarrow tSt,\;(s_1,t_1,\sigma) \longmapsto t_1$ is an isomorphism from the local submonoid $e\,T(S)\,e$ of $T(S)$ onto the local submonoid $tSt$ of $S$; if $\alpha=\tau$, then $\chi: e\,T(S)\,e \longrightarrow sSs,\;(s_1,t_1,\tau) \longmapsto s_1$ is an isomorphism from the local submonoid $e\,T(S)\,e$ of $T(S)$ onto the local submonoid $sSs$ of $S$.
\end{itemize}
\end{prop}

\begin{proof}
$(i)$. If $\alpha=\sigma$, then $t\in E(S)$ and by Proposition \ref{greenTS}.$(iii)$, 
\[H_e=\{(s_1,t_1,\sigma)\, :\; t_1\,\mathcal{H}\, t \;\mbox{ and }\; (s,s_1)=a(t,t_1) \,\mbox{ for some }\, a\in S^1\}\, .\]
It is now evident that $\psi$ is a well-defined surjective homomorphism. If $at=bt=s$ for some $a,b\in S^1$, then $at_1=bt_1$ and so $\psi$ is also injective. 

If $\alpha=\tau$, then $s\in E(S)$ and $t\mathcal{L} s$. Let $g=(s_1,t_1,\beta)\in T(S)$. Then $g\mathcal{H} e$ \iff\ $\beta=\tau$, $t_1\mathcal{H} t$ and $(s,s_1)=a(t,t_1)$ for some $a\in S$. Since $s\mathcal{L} t$, the conditions $t_1\mathcal{H} t$ and $(s,s_1)=a(t,t_1)$ are equivalent to the conditions $s_1\mathcal{H} s$ and $(t,t_1)=b(s,s_1)$ for some $b\in S^1$. Thus
\[H_e=\{(s_1,t_1,\sigma)\, :\; s_1\,\mathcal{H}\, s \;\mbox{ and }\; (t,t_1)=b(s,s_1) \,\mbox{ for some }\, b\in S^1\}\, .\]
The proof now follows similarly to the case $\alpha=\sigma$.

$(ii)$. If $\alpha=\sigma$, then $t\in E(S)$ and \[eT(S)e=\{(sut,tut,\sigma),\,(sus,tus,\sigma)\,:\; u\in S\}=\{(av,v,\sigma)\, :\; v\in tSt\}\]
for $a\in S^1$ such that $s=at$. If $\alpha=\tau$, then $s\in E(S)$ and $s\mathcal{L} t$. Let $b\in S^1$ such that $t=bs$. Then
\[eT(S)e=\{(sut,tut,\tau),\,(sus,tus,\tau)\,:\; u\in S\}=\{(v,bv,\tau)\, :\; v\in sSs\}\, .\]
It is now trivial to check that for either $\alpha=\sigma$ or $\alpha=\tau$, the mapping $\chi$ is an isomorphism.
\end{proof}

Since the subgroups and the local submonoids of $S\in\mI$ and $T(S)$ are isomorphic by the previous result, we now have:

\begin{cor}\label{sgpandlsm}
Let $S\in\mI$ and let $\mathcal{P}$ be any group property and $\mathcal{Q}$ be any semigroup property.
\begin{itemize}
\item[$(i)$] The (maximal) subgroups of $S$ have property $\mathcal{P}$ \iff\ the (ma\-xi\-mal) subgroups of $T(S)$ have property $\mathcal{P}$.
\item[$(ii)$] The local submonoids of $S$ have property $\mathcal{Q}$ \iff\ the local submonoids of $T(S)$ have property $\mathcal{Q}$.
\end{itemize}
\end{cor}

An e-variety of regular semigroups \cite{hall89,kadszendrei} is a class of these semigroups closed for homomorphic images, regular subsemigroups and direct products. We know that the local submonoids of a regular semigroup are regular too. Thus, if $\bf V$ is an e-variety of regular semigroup, we can define the class $\bf LV$ of all regular semigroups whose local submonoids belong to $\bf V$. It is well known that $\bf LV$ is again an e-variety of regular semigroups, and we call locally $\bf V$ the semigroups from $\bf LV$. Corollary \ref{sgpandlsm}.$(ii)$ now implies the next result.

\begin{cor}\label{evar}
Let $\bf V$ be an e-variety of regular semigroups. Then $S\in \bf LV$ \iff\ $T(S)\in\bf LV$.
\end{cor}

In particular, we can conclude that $S$ is a completely simple semigroup (locally a group) \iff\ $T(S)$ is a completely simple semigroup; $S$ is a combinatorial strict semigroup (locally a semilattice) \iff\ $T(S)$ is a combinatorial strict semigroup; $S$ is a strict semigroup (locally a semilattice of groups) \iff\ $T(S)$ is a strict semigroup; and $S$ is a locally inverse semigroup (locally an inverse semigroup) \iff\ $T(S)$ is a locally inverse semigroup.

The Corollary \ref{evar} cannot be obtained directly from the previously known constructions of an embedding of a semigroup $S$ into a semiband $B(S)$ since those constructions usually assume that $S$ is a monoid. For example, the usual conclusions were that $B(S)$ is completely simple \iff\ $S$ is a group, or that $B(S)$ is locally inverse \iff\ $S$ is inverse. However, for some of those constructions, we can obtain Corollary \ref{evar} if we consider instead a subsemiband $B^*(S)$ of $B(S^1)$ and embed the semigroup $S$ into $B^*(S)$.

A semigroup $S$ is simple [bisimple] if it has only one $\mathcal{J}$-class [$\mathcal{D}$-class]. A semigroup $S$ with element $0$ is $0$-simple [$0$-bisimple] if $\{0\}$ and $S\setminus\{0\}$ are the only $\mathcal{J}$-classes [$\mathcal{D}$-classes] of $S$ and $S$ is not a null semigroup, that is, $S^2\neq\{0\}$. A non-zero idempotent $e\in E(S)$ is called primitive if for all $f\in E(S)$,
\[ef=fe=f\neq 0 \;\Longrightarrow\; e=f\, .\]
It is well known that a $0$-simple [simple] semigroup has a primitive idempotent \iff\ all non-zero idempotents are primitive. A completely $0$-simple [completely simple] semigroup is a $0$-simple [simple] semigroup with a primitive idempotent. Note that we have used above that a completely simple semigroup is a regular semigroup whose local submonoids are groups. It is well known that the two definitions are equivalent. 

Let $S$ be a semigroup and $a\in S$. Let $J(a)$ be the principal ideal generated by $a$ and $J_a$ the $\mathcal{J}$-class of $a$. Then $I(a)=J(a)\setminus J_a$ is an ideal of $J(a)$ and $J(a)/I(a)\cong J_a\cup\{0\}$ is either a $0$-simple semigroup or a null semigroup. Note further that $J(a)/I(a)$ is $0$-simple \iff\ there exist $b,c\in J_a$ such that $bc\in J_a$. The semigroups $J(a)/I(a)$ are called the principal factors of $S$, and $S$ is called semisimple if all its principal factors are $0$-simple semigroups. Further, if all principal factors of $S$ are completely $0$-simple semigroups, then we say that $S$ is completely semisimple.

We need one more definition for the next result. A semigroup $S$ is [left, right] cryptic if the Green's $\mathcal{H}$-relation is a [left, right] congruence on $S$.

\begin{prop}\label{simple}
Let $\mathcal{P}$ be one of the following properties: simple, bisimple, completely simple, semisimple, completely semisimple and [left, right] cryptic. Then $S\in\mI$ has property $\mathcal{P}$ \iff\ $T(S)$ has property $\mathcal{P}$.
\end{prop}

\begin{proof}
The bisimple and simple cases follow respectively from Proposition \ref{greenTS}.$(iv)$ and $(v)$, and the completely simple case follows from Corollary \ref{evar}. Since each $\mathcal{J}$-class of $T(S)$ contains exactly one $\mathcal{J}$-class of $S\varphi$ (Corollary \ref{greenrestriction}), if $S$ is semisimple (and so $S\varphi$ is semisimple), then $T(S)$ is semisimple.
Let $J$ be a $\mathcal{J}$-class of $S$ and let $J_1$ be the $\mathcal{J}$-class of $T(S)$ containing $J\varphi$. If $T(S)$ is semisimple, then there exist $(s_i,t_i,\alpha_i)\in J_1$ for $i=1,2$ such that $(s_1,t_1,\alpha_1)(s_2,t_2,\alpha_2)\in J_1$. Thus $t_1,t_2\in J$. If $\alpha_1=\sigma$, then $t_1t_2\in J$. If $\alpha_1=\tau$, then $t_1s_2\in J$; but since $s_2\in S^1t_2$, we must have also $s_2\in J$. We conclude that if $T(S)$ is semisimple, then $S$ is semisimple too. We have shown the semisimple case.

For the completely semisimple case, we already know that $S\in\mI$ is semisimple \iff\ $T(S)$ is semisimple. Let $e\in E(S)$. Observe that $(s,t,\alpha)(e,e,\sigma)= (s,t,\alpha)$ implies $\alpha=\sigma$, and that $(e,e,\sigma)(s,t,\alpha)=(s,t,\alpha)$ implies $s=t$. It is now trivial to check that $e$ is a primitive idempotent of a principal factor of $S$ \iff\ $(e,e,\sigma)$ is a primitive idempotent of the corresponding principal factor of $T(S)$. Hence, $S$ is completely semisimple \iff\ $T(S)$ is completely semisimple.

If $T(S)$ is [left, right] cryptic, then $S$ is [left, right] cryptic since \[\mathcal{H}^{T(S)}\cap S\varphi= \mathcal{H}^{S\varphi}\]
and $S$ is isomorphic to $S\varphi$. Consider now $(s,t,\alpha),(u,v,\beta),(p,q,\gamma)\in T(S)$ such that $(s,t,\alpha)\, \mathcal{H} \,(u,v,\beta)$. Then $\alpha=\beta$, $t\mathcal{H} v$ and $(s,u)=a(t,v)$. If $S$ is left cryptic, then 
\[(p,q,\gamma)(s,t,\alpha)\,\mathcal{H}\, (p,q,\gamma)(u,v,\beta)\]
since $(p,q,\gamma)(s,t,\alpha)=(p_1t,q_1t,\alpha)$ and $(p,q,\gamma)(u,v,\beta)=(p_1v,q_1v,\alpha)$ for $(p_1,q_1)=(p,q)$ (if $\gamma=\sigma$) or $(p_1,q_1)=(pa,qa)$ (if $\gamma=\tau$). Hence, $T(S)$ is left cryptic if $S$ is left cryptic. If $S$ is right cryptic, then $tr\mathcal{H} vr$ for any $r\in S$, and so 
\[(s,t,\alpha)(p,q,\gamma)\,\mathcal{H}\, (u,v,\beta)(p,q,\gamma)\]
since $(s,t,\alpha)(p,q,\gamma)=(sr,tr,\gamma)$ and $(u,v,\beta)(p,q,\gamma)=(ur,vr,\gamma)$ for $r=p$ (if $\alpha=\tau$) or $r=q$ (if $\alpha=\sigma$). Hence, $T(S)$ is right cryptic if $S$ is right cryptic. We have shown that $S\in\mI$ is [left, right] cryptic \iff\ $T(S)$ is [left, right] cryptic. 
\end{proof}

Let $S\in\mI$ be a semigroup with $0$. Then $\overline{0}=\{(0,0,\sigma),(0,0,\tau)\}$ is the kernel of $T(S)$, that is, the minimal ideal of $T(S)$. Let $T^*(S)=T(S)/\overline{0}$. Clearly $T^*(S)$ is a semiband with $0$ and the mapping $\varphi^*:S\longrightarrow T^*(S)$ defined by
\[s\varphi^*=\left\{\begin{array}{ll}
s\varphi &\quad\mbox{ if } s\neq 0 \\ [.2cm]
\overline{0} &\quad\mbox{ if } s=0
\end{array}\right.\]
embeds $S$ into $T^*(S)$.

\begin{prop}\label{0simple}
Let $S\in\mI$ be a semigroup with $0$ element. Then $S$ is 0-simple, 0-bisimple or completely 0-simple \iff\ $T^*(S)$ is respectively 0-simple, 0-bisimple or completely 0-simple.
\end{prop} 

\begin{proof}
If $S$ is $0$-simple, then $\overline{0}$ and $T(S)\setminus\overline{0}$ are the only $\mathcal{J}$-classes of $T(S)$ by Proposition \ref{greenTS}.$(v)$, and so $T^*(S)$ is a principal factor of $T(S)$. By Proposition \ref{simple} the semiband $T(S)$ is semisimple ($S$ is semisimple), whence $T^*(S)$ is a $0$-simple semigroup. If $T^*(S)$ is $0$-simple, then $S$ has only two $\mathcal{J}$-classes, namely $\{0\}$ and $S\setminus\{0\}$, again by Proposition \ref{greenTS}.$(v)$. Note that if $S$ is a null semigroup, then $T^*(S)$ is a null semigroup too. Hence $S$ is $0$-simple. We have shown that $S$ is $0$-simple \iff\ $T^*(S)$ is $0$-simple.

A $0$-bisimple semigroup is a $0$-simple semigroup with only two $\mathcal{D}$-classes. Thus, the $0$-bisimple case follows from Proposition \ref{greenTS}.$(iv)$ and the $0$-simple case. As in the proof of the completely simple case of Proposition \ref{simple}, to show the completely $0$-simple case we need only to show that $e$ is a primitive idempotent of $S$ \iff\ $(e,e,\sigma)$ is a primitive idempotent of $T^*(S)$. This is trivial to check and therefore we leave the details to the reader.
\end{proof}

\section{Regular semigroups}\label{section4}

In this section we shall assume that $S$ is always a regular semigroup. Then $T(S)$ is also regular by Proposition \ref{finperreg}. Let $R$ be a regular subsemigroup of $T(S)$ containing $S\varphi$. Then
\[\mathcal{K}^{T(S)}\cap R=\mathcal{K}^R\]
for the Green relation $\mathcal{K}\in\{\mathcal{H}, \mathcal{L},\mathcal{R}\}$ since $R$ is a regular subsemigroup. In the next result we show that the same equality holds true for $\mathcal{K}=\mathcal{D}$ and for $\mathcal{K}= \mathcal{J}$. We shall prove also that $(s,t,\alpha)\leq^{T(S)}(u,v,\beta)$ \iff\ $(s,t,\alpha)\leq^{R}(u,v,\beta)$ for any $(s,t,\alpha),(u,v,\beta)\in R$.

\begin{prop}\label{regsubsem}
Let $S$ be a regular semigroup and let $R$ be a regular subsemigroup of $T(S)$ containing $S\varphi$. Then 
\[\mathcal{K}^{T(S)}\cap R=\mathcal{K}^R\quad\mbox{ and }\quad \leq^{T(S)}\cap R=\leq^{R}\]
for $\mathcal{K}\in\{\mathcal{H}, \mathcal{L}, \mathcal{R},\mathcal{D},\mathcal{J}\}$.
\end{prop}

\begin{proof}
As observed above, we only need to show this proposition for $\mathcal{K}\in\{\mathcal{D}, \mathcal{J}\}$. It is also clear that
\[\mathcal{K}^R\subseteq\mathcal{K}^{T(S)}\cap R\quad\mbox{ and }\quad \leq^{R}\,\subseteq\, \leq^{T(S)}\cap R\; .\]
Let $(s,t,\alpha),(u,v,\beta)\in R$ such that $(s,t,\alpha)\mathcal{J}^{T(S)} (u,v,\beta)$. Then 
\[(s,t,\alpha)\mathcal{L}^R (t,t,\alpha),\quad (u,v,\beta)\mathcal{L}^R (v,v,\beta)\quad\mbox{ and }\quad (t,t,\alpha)\mathcal{J}^{S\varphi} (v,v,\beta)\; .\]
Hence $(s,t,\alpha)\mathcal{J}^R (u,v,\beta)$, and  $\mathcal{J}^{T(S)}\cap R=\mathcal{J}^R$. The proof for $\mathcal{K}=\mathcal{D}$ is similar.

Assume now that $(s,t,\alpha)\leq^{T(S)} (u,v,\beta)$ with $(s,t,\alpha),(u,v,\beta)\in R$. By Proposition \ref{greenTS}.$(vi)$ we know that $\alpha=\beta$, $t\leq v$ and $(s,u)=a(t,v)$ for some $a\in S^1$. Since $S$ is regular, there are idempotents $e,f\in S$ such that $t=ev=vf$ and $e\mathcal{R}t\mathcal{L} f$. Let $t'$ be the inverse of $t$ such that $tt'=e$ and $t't=f$. Then $(ae,e,\sigma)=(s,t,\alpha) (t',t',\sigma)\in R$ and $(f,f,\alpha)=(t',t',\sigma) (s,t,\alpha)\in R$. Observe now that $(ae,e,\sigma)$ and $(f,f,\alpha)$ are idempotents of $R$ such that
\[(ae,e,\sigma)(u,v,\beta)=(s,t,\alpha)\quad\mbox{ and }\quad (u,v,\beta)(f,f,\alpha)=(s,t,\alpha)\, .\]
Thus $(s,t,\alpha)\leq^R (u,v,\beta)$, and $\leq^{T(S)}\cap R=\leq^{R}\,$. 
\end{proof}

Let $R(S)=\{(s,t,\alpha)\in T(S)\, :\; s\mathcal{L}t\,\}$. Since $\mathcal{L}$ is a right congruence on $S$, $R(S)$ is a subsemigroup of $T(S)$. The idempotent of $R(S)$ are the elements $(s,t,\alpha)\in R(S)$ such that either $\alpha=\sigma$ and $t\in E(S)$, or $\alpha=\tau$ and $s\in E(S)$.

\begin{prop}\label{greenRS}
Let $S$ be a regular semigroup. Then $R(S)$ is a regular semiband of depth 2 and $\varphi:S \longrightarrow R(S),\, s\longmapsto (s,s,\sigma)$ embeds $S$ into $R(S)$. Further,
\[\mathcal{K}^{T(S)}\cap R(S)= \mathcal{K}^{R(S)} \quad\mbox{ and }\quad \leq^{T(S)}\cap R(S)=\leq^{R(S)}\]
for $\mathcal{K}\in\{\mathcal{H}, \mathcal{L}, \mathcal{R},\mathcal{D},\mathcal{J}\}$.
\end{prop}

\begin{proof}
Let $s,t\in S$ such that $s\mathcal{L}t$. Let $s'$ and $t'$ be inverses of $s$ and $t$ respectively. Then $(st',tt',\sigma),\,(t't,t,\tau)\in E(R(S))$ and
\[(s,t,\tau)=(st',tt',\sigma)(t't,t,\tau)\, .\]
Similarly $(ss',ts',\tau),\,(s,s's,\sigma)\in E(R(S))$ and 
\[(s,t,\sigma)=(ss',ts',\tau)(s,s's,\sigma).\] 
Thus $R(S)$ is a semiband of depth 2. Observe now that $(t',t',\sigma)$ is an inverse of $(s,t,\sigma)$ and that $(s',s',\tau)$ is an inverse of $(s,t,\tau)$. Hence $R(S)$ is a regular semigroup. Clearly $\varphi$ embeds $S$ into $R(S)$ and the second statement of this proposition is a particular case of Proposition \ref{regsubsem}.
\end{proof}

In the previous section we showed that $T(S)$ preserves many properties of $S$. We can now show that $R(S)$ preserves the same properties. Some of those properties follow immediately from the fact that $S\varphi\subseteq R(S)\subseteq T(S)$ and Proposition \ref{greenRS}, but for others we have to mimic the proof presented in the previous section for $T(S)$. In the next three result we register those properties preserved by $R(S)$, but we shall not include their proofs since is just a question of rephrasing the proofs presented in the previous section.

\begin{prop}
Let $\mathcal{O}$ be any group property and let $\mathcal{Q}$ be any semigroup property. Let $\mathcal{P}$ be one of the following properties: finite, periodic, simple, bisimple, completely simple, semisimple, completely semisimple, [left, right] cryptic, the (maximal) subgroups have property $\mathcal{O}$ or the local submonoids have property $\mathcal{Q}$. Then, a regular semigroup $S$ has property $\mathcal{P}$ \iff\ the regular semigroup $R(S)$ has property $\mathcal{P}$.
\end{prop}

\begin{prop}
Let $\bf V$ be an e-variety of regular semigroup. Then $S\in\bf LV$ \iff\ $R(S)\in\bf LV$.
\end{prop}

If $S$ is a semigroup with a $0$-element, then we can consider also the semigroup $R^*(S)= R(S)/\overline{0}$. We can easily see that $S$ embeds naturally into $R^*(S)$ and we can obtain a version for $R^*(S)$ of Proposition \ref{0simple}.

\begin{prop}
Let $S$ be a regular semigroup with $0$-element. Then $S$ is 0-simple, 0-bisimple or completely 0-simple \iff\ $R^*(S)$ is respectively 0-simple, 0-bisimple or completely 0-simple.
\end{prop}

There are however properties of $S$ preserved by $R(S)$ that are not preserved by $T(S)$. One of those properties is the complete regularity. For example, if $S=\{0,1\}$ with the usual product, then $(0,1,\tau)\in T(S)$ and $(0,1,\tau)^2=(0,0,\tau)$ is not $\mathcal{H}$-related to $(0,1,\tau)$. Hence $T(S)$ is not completely regular. We shall prove next that if $S$ is completely regular, then $R(S)$ is also completely regular.

\begin{prop}
A regular semigroup $S$ is completely regular \iff\ $R(S)$ is completely regular.
\end{prop} 

\begin{proof}
We just have to prove the direct implication since a regular subsemigroup of a completely regular semigroup is again completely regular. Let $s,t\in S$ such that $s\mathcal{L}t$, that is, $s=at$ and $t=bs$ for some $a,b\in S$. Let $e\in H_s\cap E(S)$ and $f\in H_t\cap E(S)$. Then $be\in H_t$ and $abe=abss^{-1}=ats^{-1}= ss^{-1}=e$ where $s^{-1}$ is the inverse of $s$ in the group $H_s$. It is now clear that $(af,f,\sigma)$ and $(e,be,\tau)$ are idempotents of $R(S)$. Further $(af,f,\sigma) \mathcal{H} (s,t,\sigma)$ and $(e,be,\tau)\mathcal{H} (s,t,\tau)$. Thus $R(S)$ is completely regular.
\end{proof}

The previous result allows to conclude that every (finite) complete regular semigroup is embeddable into a (finite) complete regular semiband. This result was first proved by Pastijn \cite{pastijn77} using a subsemiband of $A(S)$ (see also \cite{almeidamoura} for another proof using Petrich's embeding \cite{petrich04} instead). If $S$ is a completely regular monoid, then $S$ is a semilattice $Y$ of completely simple semigroups $D_\alpha$, $\alpha\in Y$. Pastijn \cite[Theorem 3.5]{pastijn77} showed that the subset
\[A_1=\{(\ov{s}\ov{h}\ov{t})\rho,\,(\ov{s}\ov{h}\ov{t}\ov{h})\rho\,:\; s\in D_\nu,\; t\in D_\mu,\; \nu,\mu\in Y,\;\nu\geq\mu\,\}\, .\]
of $A(S)$ is a complete regular subsemiband. The mapping $\psi_1:\,s \longrightarrow (\ov{1}\ov{h}\ov{s})\rho$ embeds $S$ into $A_1$ (in \cite{pastijn77} was used the embedding $s\longrightarrow (\ov{1}\ov{h}\ov{s} \ov{h})\rho$ instead). Let $s\in D_\nu$ and $t\in D_\mu$ with $\nu\geq\mu$ and let $s_1=stt^{-1}=st^{-1}t$. Then $s_1\mathcal{L} t$, and $(\ov{s}\ov{h}\ov{t})\rho= (\ov{s_1}\ov{h}\ov{t})\rho $ and $(\ov{s}\ov{h}\ov{t}\ov{h})\rho=(\ov{s_1}\ov{h}\ov{t}\ov{h})\rho$. Hence
\[A_1=\{(\ov{s}\ov{h}\ov{t})\rho,\,(\ov{s}\ov{h}\ov{t}\ov{h})\rho\,:\; s\mathcal{L}t\,\}\, .\]
It is now trivial to check that $\psi^{-1}_{|A_1}: A_1\longrightarrow R(S)$ is an isomorphism for $\psi^{-1}:A(S)\longrightarrow T(S)$ the isomorphism defined at the end of Section \ref{section3}. Furthermore, $\psi_1\psi^{-1}=\varphi$ for $\varphi:S\longrightarrow R(S)$ the embedding defined in Proposition \ref{greenRS}. In particular, $A_1$ is a semiband of depth 2 (in \cite{pastijn77} it was shown that every element of $A_1$ is the product of at most 4 idempotents).

Let $\mathcal{T}_X$ be the full transformation semigroup on a set $X$ and let $X'=\{x'\,:\; x\in X\}$ be a set disjoint from $X$ but with the same cardinality. Recall the Higgins' embedding \cite{higgins95} of $\mathcal{T}_X$ into a subsemiband $T$ of $\mathcal{T}_Y$ for $Y=X\cup X'$ (see Section 2). Next, we shall compare the two semibands $R(\mathcal{T}_X)$ and $T$ of depth 2. It is well known that for $\lambda,\mu\in\mathcal{T}_X$, $\lambda\mathcal{L}\mu$ \iff\ $X\lambda=X\mu$. Thus
\[R(\mathcal{T}_X)=\{(\lambda,\mu,\alpha)\in \mathcal{T}_X\times\mathcal{T}_X\times R_2\, :\;X\lambda=X\mu\,\}\,.\]
For $\Delta=(\lambda,\mu,\alpha)\in R(\mathcal{T}_X)$ 
define $\overline{\delta}\in\mathcal{T}_Y$ as follows:
\[\left\{\begin{array}{l}
x\overline{\delta}=x\mu \\ [.2cm]
x'\overline{\delta}=x\lambda
\end{array}\right.\;\mbox{ if } \alpha=\sigma\qquad \mbox{ or }\qquad
\left\{\begin{array}{l}
x\overline{\delta}=(x\mu)' \\ [.2cm]
x'\overline{\delta}=(x\lambda)'
\end{array}\right.\;\mbox{ if } \alpha=\tau\]
Then $\overline{\delta}\in T$ since $X\lambda=X\mu$. Further, $X\overline{\delta}=X'\overline{\delta} \subseteq X$ if $\alpha=\sigma$, and $X\overline{\delta}=X'\overline{\delta} \subseteq X'$ if $\alpha=\tau$. Consider now the mapping $\psi: R(\mathcal{T}_X)\longrightarrow T$ defined by $\Delta\psi=\overline{\delta}$.

\begin{prop}
The semibands $R(\mathcal{T}_X)$ and $T$ are isomorphic and the mapping $\psi$ defined above is an isomorphism from $R(\mathcal{T}_X)$ onto $T$.
\end{prop}

\begin{proof}
The mapping $\psi$ is clearly a bijection. So, we need to show only that $\psi$ is a homomorphism. Let $\Delta_1=(\lambda_1,\mu_1,\alpha_1)$ and $\Delta_2= (\lambda_2,\mu_2,\alpha_2)$ be two elements of $R(\mathcal{T}_X)$, and let $\Delta=\Delta_1\Delta_2$. Let $\overline{\delta_1}=\Delta_1\psi$, $\overline{\delta_2}=\Delta_2 \psi$ and $\overline{\delta} = \Delta\psi$. We want to prove that $\overline{\delta}= \overline{\delta_1}\, \overline{\delta_2}$. If $\alpha_1=\sigma$ and $\alpha_2= \sigma$, then $\Delta=(\lambda_1\mu_2,\mu_1 \mu_2,\sigma)$. Further, for $x\in X$,
\[x\,\overline{\delta_1}\,\overline{\delta_2}= x\mu_1 \overline{\delta_2}= x\mu_1\mu_2= x\overline{\delta} \quad\mbox{ and }\quad x'\,\overline{\delta_1}\, \overline{\delta_2}= x'\lambda_1 \overline{\delta_2}= x'\lambda_1\mu_2=x'\overline{\delta}\, .\]
Thus $\overline{\delta}=\overline{\delta_1} \overline{\delta_2}$ if $\alpha_1=\alpha_2=\sigma$. The other three cases are shown similarly. 
\end{proof}

Let $C$ be a class of semigroups. For $n\geq 2$ let $\sigma_C^{(2)} (n)$ denote the smallest integer $k\geq n$ such that every semigroup of $C$ of order not greater than $n$ can be embeddable into a semiband of depth 2 and order not greater than $k$. Let $G$ be the class of all groups and let $Reg$ be the class of all regular semigroups. Giraldes and Howie \cite{giraldeshowie} showed that
\[\sigma_G^{(2)}(n)\leq 2n^2\qquad\mbox{ and }\qquad \sigma_{Reg}^{(2)}(n)\leq (n+1)^3\, .\] 
We can now improve the upper bound for $\sigma_{Reg}^{(2)}(n)$.

\begin{prop}
$\sigma_{Reg}^{(2)}(n)\leq 2n^2\,$ for all $n\geq 2$.
\end{prop}

\begin{proof}
We just need to observe that $R(S)$ has order less than or equal to $2n^2$ if $S$ has order $n$.
\end{proof}

Let $S$ be regular semigroup of order $n$. Note that $R(S)$ has order $2n^2$ only when $S$ is a left group; but if $S$ is not a left group, then the order of $R(S)$ can decrease significantly. Note further that we can do a left-right dual construction of $R(S)$ and obtain another semiband $L(S)$ of depth 2 in which $S$ can be embedded. We can check also that $L(S)$ has order $2n^2$ only when $S$ is a right group. Therefore, if $S$ is a non-group regular semigroup of order $n$, then $S$ can be embedded in a regular semiband of depth 2 of order less than $2n^2$. It seems reasonable to define now the integer 
\[\sigma_C^{(2)}(n,m)\]
as follows. For a class $C$ of regular semigroups and an integer $n\geq 2$, let $\sigma_C^{(2)}(n,m)$ denote the smallest integer $k\geq n$ such that every semigroup of $C$ of order not greater than $n$ and with subgroups of order not greater that $m$, can be embedded into a semiband of depth 2 of order not greater than $k$. Of course that $\sigma_C^{(2)}(n,m)$ makes sense only for classes $C$ of regular semigroups containing other semigroups besides groups.

\begin{prop}
Let $C$ be a class of regular semigroups. Then
\[\sigma_C^{(2)}(n,m)\leq 2n\sqrt{n}\sqrt{m}= 2\sqrt{n^3m}\, .\]
\end{prop}

\begin{proof}
Let $S\in C$ be a regular semigroup of order not greater than $n$ and with subgroups of order not greater than $m$. If $l$ and $r$ are respectively the sizes of the largest $\mathcal{L}$-class and the largest $\mathcal{R}$-class of $S$, then $R(S)$ has order less than $2nl$ and $L(S)$ has order less than $2nr$. Note that $lr\leq mn$, and for $h=\min\{l,r\}$, $h\leq \sqrt{mn}$. Hence $S$ is embeddable into a semiband of depth 2 and order not greater than $2nh\leq 2n\sqrt{n}\sqrt{m}$.
\end{proof}

\vspace*{.5cm}

\noindent{\bf Acknowledgments}: This work was partially funded by the European Regional Development Fund through the programme COMPETE and by the Portuguese Government through the FCT -- Funda\c{c}\~ao para a Ci\^encia e a 
Tecnologia under the project PEst-C/MAT/UI0144/2011.

\begin{flushleft}
Lu\'\i s Oliveira \\
Departamento de Matem\'atica Pura,\\
Faculdade de Ci\^encias da Universidade do Porto,\\
R. Campo Alegre, 687, 4169-007 Porto, Portugal\\
e-mail: loliveir@fc.up.pt
\end{flushleft}

\end{document}